\documentclass[12pt]{amsart}
\usepackage{stmaryrd}
\usepackage{textcomp}
\usepackage{enumerate}
\usepackage{setspace}
\usepackage{amssymb}
\usepackage{amsthm}
\usepackage{amsmath}
\usepackage{graphicx}
\usepackage{extarrows}
\usepackage{marvosym}
\usepackage{empheq}
\usepackage{latexsym}
\usepackage{endnotes}
\usepackage{fontenc}
\usepackage{color}
\usepackage{comment}

\allowdisplaybreaks

\begin{document}

	\newtheorem{theorem}{Theorem}
	\newtheorem{definition}[theorem]{Definition}
	\newtheorem{question}[theorem]{Question}
	
	\newtheorem{lemma}[theorem]{Lemma}
	\newtheorem{claim}[theorem]{Claim}
	\newtheorem*{claim*}{Claim}
	\newtheorem{corollary}[theorem]{Corollary}
	\newtheorem{problem}[theorem]{Problem}
	\newtheorem{conj}[theorem]{Conjecture}
	\newtheorem{fact}[theorem]{Fact}
	\newtheorem{observation}[theorem]{Observation}
	\newtheorem{prop}[theorem]{Proposition}

	\newcommand\marginal[1]{\marginpar{\raggedright\parindent=0pt\tiny #1}}
	
	\theoremstyle{remark}
	\newtheorem{example}[theorem]{Example}
	\newtheorem{remark}[theorem]{Remark}
	\newcommand{\RED}{\color{red}}
	
	\newcommand{\supp}{\mathrm{supp}}
	\def\eps{\varepsilon}
	\def\HH{\mathcal{H}}
	\def\E{\mathbb{E}}
	\def\C{\mathbb{C}}
	\def\R{\mathbb{R}}
	\def\Z{\mathbb{Z}}
	\def\N{\mathbb{N}}
	\def\PP{\mathbb{P}}
	\def\T{\mathbb{T}}
	\def\l{\lambda}
	\def\s{\sigma}
		\def\vp{\varphi}
	\def\t{\theta}
	\def\a{\alpha}
	\def\la{\langle}
	\def\ra{\rangle}	
	\def\endproof{{\hfill $\square$} }
	\def\Xt{\widetilde{X}}
	\def\Pt{\widetilde{P}}
	\def\Var{\mathrm{Var}}
	\def\PV{\mathrm{PV}}
	\def\NV{\mathrm{NV}}
	\def\NN{\mathcal{N}}
	\def\CC{\mathcal{C}}
	
	\def\cA{\mathcal{A}}
	\def\cQ{\mathcal{Q}}	
	\def\cC{\mathcal{C}}
	\def\cG{\mathcal{G}}
	\def\F{\mathcal{F}}
	\def\tm{\tilde{\mu}}
	\def\ts{\tilde{\sigma}}
	\def\L{\Lambda}
	\def\z{\zeta}
		\def\p{\partial}
	\def\g{\gamma}
	\def\Q{\mathcal{Q}}
	\newcommand{\Cov}{\mathrm{Cov}}
	\renewcommand{\P}{\mathbb{P}}
	\newcommand{\Corr}{\mathrm{Corr}}
	\newcommand{\SCS}{\mathrm{SCS}}
	\newcommand{\SCC}{\mathrm{SCC}}
	\newcommand{\one}{\mathbf{1}}
	\newcommand{\bx}{\mathbf{x}}
	\newcommand{\by}{\mathbf{y}}
		\newcommand{\bu}{\mathbf{u}}
		\newcommand{\bv}{\mathbf{v}}
	\newcommand{\bw}{\mathbf{w}}
	\newcommand{\bz}{\mathbf{z}}
	\newcommand{\bt}{\mathbf{t}}
	\newcommand{\mes}{\mathrm{mes}}
	\newcommand{\tr}{\mathrm{tr}}
	\newcommand{\Exp}{\mathrm{Exp}}
	\newcommand{\br}{\mathbf{r}}
		\def\im{\mathrm{Im}}
		\def\re{\mathrm{Re}}
\pagestyle{plain}

\title{Real roots near the unit circle of random polynomials}
\author{Marcus Michelen}
\address{The University of Illinois at Chicago. Department of Mathematics, Statistics and Computer Science}
\email{michelen.math@gmail.com}

\begin{abstract}
Let $f_n(z) = \sum_{k = 0}^n \eps_k z^k$ be a random polynomial where $\eps_0,\ldots,\eps_n$ are i.i.d.  random variables with $\E \eps_1 = 0$ and  $\E \eps_1^2 = 1$. Letting $r_1, r_2,\ldots, r_k$ denote the real roots of $f_n$, we show that the point process defined by $\{|r_1| - 1,\ldots, |r_k| - 1 \}$ converges to a non-Poissonian limit on the scale of $n^{-1}$ as $n \to \infty$.  Further, we show that for each $\delta > 0$, $f_n$ has a real root within $\Theta_{\delta}(1/n)$ of the unit circle with probability at least $1 - \delta$.  This resolves a conjecture of Shepp and Vanderbei from 1995 by confirming its weakest form and refuting its strongest form.
\end{abstract}

\maketitle

\section{Introduction}

Our main object of study is the set of real roots of the random polynomial
\begin{equation}\label{eq:f-def}
f_n(z) = \sum_{k = 0}^n \eps_k z^k
\end{equation}
with $\eps_k$ i.i.d.\ with $\E \eps_1 = 0$ and $\E \eps_1^2 =1$.

The behavior of real roots of the random polynomial $f_n$ has been studied intensely over the past century.  Even further back, the issue of the number of real roots of a random---or ``typical''---polynomial was pondered as early as 1782 by Waring (see \cite{kostlan-chapter}).  Among the first rigorous results on this topic were by Bloch and P\'olya \cite{bloch-polya} which stated that the average number of real roots is  $O(n^{1/2})$ when the coefficients $\eps$ are i.i.d.\ and uniformly distributed in $\{-1,0,1\}$.  In an  influential sequence of papers, Littlewood and Offord \cite{littlewood-offord1,littlewood-offord2,littlewood-offord4,littlewood-offord3} showed that typically a random polynomial has at most $O(\log^2(n))$ and at least $\Omega(\log(n)/\log \log \log(n))$ real roots for various distributions of coefficients.  This series of work  included the cases when the coefficients are i.i.d.\ and taken uniformly from $\{-1,1\}$ which ultimately led to Littlewood-Offord theory.  

Shortly after Littlewood-Offord, Kac \cite{kac1943average} found an exact formula for the number of real roots when the coefficients are standard (real) Gaussians and subsequently showed that if $N_\R$ is the number of real roots of $f$ then 
\begin{equation} \label{eq:expected-real-roots}
\E N_\R \sim \frac{2}{\pi}\log n\,.
\end{equation} 
Independently, Rice \cite{rice1,rice2} found a similar exact formula for the expected number of times a sufficiently nice stochastic process crosses a given value.  The machinery build by Kac and Rice---variously called Kac formulas, Rice formulas or Kac-Rice formulas---reduces the problem of computing the expected number of real roots to analyzing the asymptotic behavior of an integral.  However, the Kac-Rice formula is most tractable in the case when the coefficients are Gaussian and typically only provides useful information when the coefficients are absolutely continuous with respect to Lebesgue measure.   From here, this asymptotic result was generalized for various different cases of coefficients, beginning with Kac who proved \eqref{eq:expected-real-roots} when the coefficients are uniform in $[-1,1]$ using a Kac-Rice formula \cite{kac1949}.  Erd\H{o}s and Offord proved \eqref{eq:expected-real-roots} when the coefficients are uniform in $\{-1,1\}$ by counting sign changes in small intervals \cite{erdosOfford}.  This approach was generalized by Stevens \cite{stevens} to a wider family of coefficients.  The study of average number of real roots culminated in a series of papers by Ibragimov and Maslova \cite{ibragimov-maslova4,ibragimov-maslova1,ibragimov-maslova3,ibragimov-maslova2} which proved \eqref{eq:expected-real-roots} is \emph{universal} in the sense that it holds for all mean-zero distributions with mild assumptions. For a modern view of \eqref{eq:expected-real-roots} along with lower order terms, see \cite{edelman-kostlan}. 

Beyond examining only the average number of real roots, Maslova showed a central limit theorem for the number of real roots \cite{maslova}.  The lower tail for the number or real roots was studied by Dembo, Poonen, Shao and Zeitouni \cite{dembo-etal} who found the asymptotic probability that $f$ has exactly $k$ real roots for a wide class of coefficients.  

Moving away from the real axis, \v{S}paro and \v{S}ur \cite{sparoSur} showed the well-known result that most roots of $f$ are near the unit circle.  Beyond results of this form, little work has been done on finer information on the \emph{location} of real roots.   In 1995, Shepp and Vanderbei \cite{shepp-vanderbei} made the following conjecture in this direction:

\begin{conj}
	Let $\eps_j$ be i.i.d.\ standard Gaussians.  Then $f_n$ has a real root with absolute value $1 + O(n^{-1})$ with high probability, meaning that for each $\delta > 0$, $f_n$ has a real root of absolute value $1 + O_\delta(n^{-1})$ with probability at least $1 - \delta$.  Further, if we let $R_n$ to be the set of real roots of $f_n$, then the point process \begin{equation} \label{eq:PP-def}
	\{n(|r| - 1) : r \in R_n \}
	\end{equation} 
	converges in distribution to a Poisson process.  
\end{conj}

We provide a full resolution to both conjectures of Shepp and Vanderbei: indeed, there is a root in $1 + O(n^{-1})$ with high probability, but the point process defined by \eqref{eq:PP-def} \emph{does not} converge to a Poisson process.  Further, we do so for a more general class of coefficients.

\begin{theorem}\label{th:main}
	Let $\{\eps_k\}$ be i.i.d.\ random variables with $\E \eps_1 = 0$ and $\E \eps_1^2 =1$.  Then the point process $\{n(|r| - 1) : r \in R_n \}$ converges in distribution to a non-Poissonian limit.  Further, for each $\delta > 0$ there exists an $M > 0$ so that $$\P(R_n \cap [1 - M/n, 1 + M/n] \neq \emptyset) \geq 1 - \delta$$ for all $n$ sufficiently large.
\end{theorem}

For an intuitive reason to understand why the limit is non-Poissonian, the roots of $f_n$ within distance $O(1/n)$ repel each other and so knowing there is a root within $O(1/n)$ of $1$ greatly reduces the likelihood of there being another.

In a paper by the author with Julian Sahasrabudhe \cite{michelen-sahasrabudhe}, we answered a question of Shepp and Vanderbei about the \emph{complex} roots of $f_n$ in the case where $\eps_j$ are i.i.d. standard Gaussians.  There, we show that the point process defined by $\{n^2 (|\zeta| - 1) : \zeta \in Z_n  \}$ where $Z_n$ is the set of roots of $f_n$ converges to the Poisson process of intensity $1/12$.  In this light, Theorem \ref{th:main} demonstrates a stark contrast between the behavior of real roots and complex roots near the unit circle.
\subsection{Overview and Organization}

In order to show convergence of the point process $\{n(|r| - 1) : r \in R_n  \}$, we show that if we define $g_n(x) := \frac{1}{\sqrt{n}} f_n(1 + x/n)$, then $g_n$ converges in distribution to a Gaussian process $g$.  We show that convergence occurs in a strong enough sense in order to determine that the roots of $g_n$---which are the same as the roots of $f_n$ up to an affine transformation---converge weakly to the zeros of $g$.  While this only appears to handle the case of \emph{positive} real roots, it will also be the case that the negative real roots captured by the function $\tilde{g}_n(x) := \frac{1}{\sqrt{n}} f_n(-1-x/n)$  converge to a copy of $g$; further, the processes $g_n$ and $\tilde{g}_n$ are asymptotically independent.  For simplicity, we first show convergence of the point process defined only by the positive roots and then upgrade these results and show convergence of the positive and negative roots simultaneously.  Convergence of the process $g_n$ is carried out in Section \ref{sec:process} and convergence of the positive roots is the content of Section \ref{sec:zeros}.   The details showing convergence of positive and negative roots simultaneously is taken care of in Appendix \ref{app:negative}.

To show that there is a real root near $1$ with high probability, we look to the limiting process $g$: if $x/y$ is large, then $g(x)$ and $g(y)$ are quite weakly correlated.  This will mean that for $M$ large enough, we will be able to find many points in $[-M,M]$ on which $g$ looks close to a centered multivariate Gaussian with independent coordinates.  It then follows that $g$ exhibits a sign change among these points with high probability.  Using convergence of $g_n$ to the limit $g$ then shows that there must be a root near $1$ with high probability.  This is shown in Section \ref{sec:near-1}.

Finally, we show that the limiting process is not Poissonian in Section \ref{sec:not-Poisson}.  The key idea here is that the zeros of $g$ repel each other while the zeros of a Poisson process do not.  To capture this repulsive property we use the Kac-Rice formula for the second combinatorial moment to show that for a small interval $I$ around zero, the probability that $g$ has two zeros in $I$ is $O(|I|^3)$ while the probability it has a zero in that interval is $\Theta(|I|)$.  This will show that the point process defined by the positive roots is not Poissonian.  Using Raikov's Theorem---i.e.\ the fact that if $X$ and $Y$ are independent then $X + Y$ is a shifted Poisson if and only if $X$ and $Y$ are---shows that the limiting process is not Poisson.

\subsection{Notation}

For non-negative integers $a$ and $b$ we write $(a)_b := a(a-1)\cdots (a-b+1)$ for the falling factorial where we interpret $(a)_0 = 1$.  For a set $U \subset \R$ we write $\mathcal{C}(U)$ for the set of continuous functions from $U$ to $\R$.  For an interval $I$ we write $|I|$ for the Lebesgue measure of $I$.

We make frequent use of various asymptotic notation.  We write $f = O(g)$ if there is an implicit constant $C$ so that $|f| \leq C |g|$; similarly we write $f = \Omega(g)$ if $|f| \geq c |g|$ for some $c > 0$; finally we write $f = \Theta(g)$ if $f = O(g)$ and $f = \Omega(g)$.  If the implicit constant depends on a parameter $\eps$ then we may write $f = O_\eps(g)$ to emphasize the dependence.  If a sequence $a_n$ tends to zero as $n \to \infty$, we write $a_n = o(1)$.  We extend this notation for real parameters: if $f(\alpha)$ is a function with $f(\alpha) \to 0$ as $\alpha \to \infty$ we write $f = o_{\alpha \to \infty}(1)$.

\section{A root near $1$} \label{sec:near-1}

To show that with high probability there is a root near $1$, we will need a multivariate central limit theorem for $g_n$ evaluated at any finite collection of points; to show convergence of the process $g_n$, we will also need control over its first two derivatives.  In order to prove this central limit theorem, we first find the limiting covariance:
\begin{lemma}\label{lem:covariance-calc}
	For each fixed $x,y \in \R$ and integers $i,j \geq 0$, we have $$\lim_{n \to \infty} \E g_n^{(i)}(x) g_n^{(j)}(y) = \int_0^1 t^{i + j} e^{(x+y)t}\,dt\,.$$
\end{lemma}
\begin{proof}
	We have \begin{align*}
	\E g_n^{(i)}(x) g_n^{(j)}(y) &= \frac{1}{n^{1+i+j}}\sum_{r = 0}^n (r)_i (r)_j (1 + x/n)^r (1 + y/n)^r  \\
	&= \frac{1}{n} \sum_{ r= 0}^n (r/n)^{i+j} e^{(x+y)r/n} + O(1/n) \\
	&\to \int_0^1 t^{i+j} e^{(x+y)t}\,dt\,.
	\end{align*}
\end{proof}

Define $g$ to be the mean-zero Gaussian process on $\R$ defined by covariance $$ \E g(x) g(y) = \int_0^1 e^{(x + y)t}\,dt\,.$$

We first show that covariance function indeed defines a Gaussian process:

\begin{lemma}\label{lem:g-exists}
	The process $g$ exists and is smooth almost surely.
\end{lemma}
\begin{proof}
	By Kolmogorov's extension theorem \cite[Section 1.2]{adler2009random}, in order to show that $g$ exists it is sufficient to show that for each finite collection of distinct points $\{x_j\}_{j = 1}^k$ the matrix $\Sigma:=(\int_0^1 e^{(x_i + x_j)t}\,dt )_{i,j}$ is positive definite.  Since the functions $e^{x_i t}$ are linearly independent elements of $L^2([0,1])$, the matrix $\Sigma$ is a Gram matrix of linearly independent elements and thus is positive definite.  Since the covariance function $(x,y) \mapsto \E g(x) g(y) = \int_0^1 e^{(x+y)t}\,dt$ is smooth, an almost surely smooth version of $g$ exists \cite[p.\ 30]{azais2009level}.
\end{proof}

Lemmas \ref{lem:covariance-calc} and \ref{lem:g-exists} together show that the covariances of $g_n$ and its derivatives converge to those of $g$ as $n \to \infty$.  This in fact will show convergence of finite dimensional distributions of $g_n$ and its first two derivatives to those of $g$.

\begin{lemma}\label{lem:fidi}
	Let $x_1,\ldots, x_k \in \R$ be distinct points.  Then as $n \to \infty$ the vector $(g_n^{(j)}(x_i) )_{i \in [k], 0 \leq j \leq 2}$ converges in distribution to the multivariate Gaussian $(g^{(j)}(x_i) )_{i \in [k],0 \leq j \leq 2}$.
\end{lemma}
\begin{proof}
	By the Cram\'er-Wold theorem (e.g., \cite[Theorem 3.9.5]{durrett}), it is sufficient to show that for all constants $c_{i,j}$ we have convergence in distribution $$\sum_{i,j} c_{i,j} g_n^{(j)}(x_i) \implies \sum_{i,j} c_{i,j} g^{(j)}(x_i)\,.$$
	Write \begin{equation}
\sum_{i,j} c_{i,j} g_n^{(j)}(x_i) =	\frac{1}{\sqrt{n}}\sum_{r = 0}^n \eps_r \left(\sum_{i,j} c_{i,j}(r)_i (1 + x_i/n)^r \right)\,.
	\end{equation}
	By the Lindeberg-Feller central limit theorem (e.g., \cite[Theorem 3.4.5]{durrett}), this converges to a mean-zero Gaussian variable, so all that is left to show is that its variance is the correct variance.  To do so, it is sufficient to show that the covariance matrix of $(g_n^{(j)}(x_i) )_{i \in [k], 0 \leq j \leq 2}$ converges to that of $(g^{(j)}(x_i) )_{i \in [k], 0 \leq j \leq 2}$; this follows from Lemma \ref{lem:covariance-calc}.
\end{proof}

In order to show that with high probability there is a root near $1$ we will show that for each $k$ we can find $k$ points $x_1,\ldots,x_k$ for which the values $g(x_j)$ have nearly independent sign.  Since Lemma \ref{lem:fidi} shows that the vector $(g_n(x_j))_{j = 1}^k$ converges to a mean-zero multivariate Gaussian, this will mean that for large $n$ the signs of $g_n(x_j)$ will roughly correspond to independent fair coin flips provided the limiting Gaussian $(g(x_j))_{j=1}^k$ has low correlations between coordinates.  This will show that there is a root with probability at least  $1 - O(2^{-k})$; taking $k$ large will show the first part of Theorem \ref{th:main}.

We now formalize that if a mean-zero multivariate Gaussian has low correlations then the signs of the coordinates are roughly independent.

\begin{lemma}\label{lem:sign-change-gauss}
	For each $k \in \N$ there is an $\eps = \eps(k) > 0$ so that the following holds.  Let $(Z_1,\ldots,Z_k)$ be a mean-zero multivariate Gaussian so that $\E Z_j^2 = 1$ for all $j$ and $|\E Z_i Z_j| \leq \eps$ for $i \neq j$.  Then $$\P(\text{All }Z_j \text{ have the same sign} ) \leq 2^{-k+2}\,.$$
\end{lemma}
\begin{proof}
	Note that if we set $(W_j)_{j=1}^k$ to be i.i.d.\ standard Gaussians, then the probability all $W_j$'s are the same sign is exactly $2^{-k+1}$.  The total variation distance between the law of $(Z_j)_{j=1}^k$ and the law of $(W_j)_{j=1}^k$ tends to zero as $\eps \to 0$ (see, e.g., \cite{devroye2018total}).  Taking $\eps$ small enough so that the total variation is at most $2^{-k+1}$ completes the proof.
\end{proof}

We now prove the second part of Theorem \ref{th:main}:
\begin{lemma} \label{lem:root-near-one}
	For each $\delta > 0$ there exists an $M > 0$ so that $$\P(R_n \cap [1 - M/n, 1 + M/n] \neq \emptyset) \geq 1 - \delta$$ for all $n$ sufficiently large.
\end{lemma}
\begin{proof}
	Let $x \geq 1$ and $y = \alpha x$ with $\alpha > 1$.  Then note that \begin{align*}
	\Cov\left(\frac{g(x)}{\sqrt{\Var\ g(x)}}, \frac{g(y)}{\sqrt{\Var\ g(y)}}\right) &= \left(\frac{e^{x+y}-1}{(x+y)^2} \right)^2\frac{4 xy}{(e^{2x} - 1)(e^{2y}-1)} \\
	&\leq \frac{16xy}{(x+y)^2} \\
	&= \frac{16 \alpha}{(1 + \alpha)^2} \\
	&= o_{\alpha \to \infty}(1)\,.
	\end{align*} 
	
	Fix $\delta > 0$.  Pick $k \in \N$ so that $2^{-k+3} \leq \delta$ and let $\eps = \eps(k)$ be as in Lemma \ref{lem:sign-change-gauss}.  Then we may choose $\alpha$ large enough so that \begin{equation}
	\Cov\left(\frac{g(x)}{\sqrt{\Var\ g(x)}}, \frac{g(y)}{\sqrt{\Var\ g(y)}}\right) \leq \eps
	\end{equation} 
	for all $x \geq 1$ and $y = \alpha x$.  Set $x_j = \alpha^{j-1}$ for $j = 1,\ldots, k$, and define the vectors $(Z_j^{(n)})_{j = 1}^k$ via $$
	(Z_j^{(n)})_{j = 1}^k = \left(\frac{g_n(x_j)}{\sqrt{\Var\ g_n(x_j)}}\right)_{j = 1}^k
	$$
	and $$(Z_j)_{j = 1}^k = \left(\frac{g(x_j)}{\sqrt{\Var\ g(x_j)}}\right)_{j = 1}^k\,.$$
	
	Then by Lemma \ref{lem:sign-change-gauss}, the probability that all $Z_j$'s have the same sign is at most $2^{-k+2}$.  By Lemma \ref{lem:fidi}, the probability that all $Z_j^{(n)}$'s have the same sign is thus at most $2^{-k+3} \leq \delta$ for $n$ large enough.
	
	For $M := \alpha^k$ we then have \begin{align*}
	\P( R_n \cap [1-M/n,1 + M/n] = \emptyset) &\leq \P(\text{All }Z_j^{(n)} \text{ have the same sign}) \\
	&\leq 2^{-k+3} \\
	&\leq \delta\,.
	\end{align*}  
\end{proof}

\section{Convergence of the Gaussian process} \label{sec:process}

The goal of this section is to show convergence of the triple $(g_n,g_n',g_n'')$ to $(g,g',g'')$; this will be useful for showing the first part of Theorem \ref{th:main}.
\begin{lemma}\label{lem:convergence-function}
	For each $M > 0$ the functions $(g_n,g_n',g_n'')$ converge in distribution on the space of continuous functions from $[-M,M]^3 \to \R^3$ endowed with the uniform topology to $(g,g',g'')$.  
\end{lemma}

In order to show convergence in distribution, there are two pieces: finite dimensional distributions converge to those of the limit and tightness of the sequence of random functions.  The convergence of finite dimensional distributions is taken care of by Lemma \ref{lem:fidi} and so all that remains is to show tightness.  Since the variables $(g_n,g_n',g_n'')$ are defined on a product space, we recall a basic fact about tightness.

\begin{fact}\label{fact:tight}
	Let $(\Omega,\mathcal{F})$ be a measurable space.  A family of probability measures $(\mu_\alpha)_\alpha$ on the product space $(\Omega^3,\mathcal{F}^{3})$ is tight provided each family of marginals is tight.
\end{fact}
\begin{proof}
	For each measure $\mu_\alpha$ write $\mu_\alpha^{(1)}$ to be the marginal $\mu_\alpha^{(1)}(B) = \mu_\alpha(B \times \Omega \times \Omega)$ and define $\mu_\alpha^{(j)}$ for $j = 2,3$ analogously.  Fix $\eps > 0$; then by tightness there exist compact sets $K_j$ so that $\mu_\alpha^{(j)}(K_j^c) \leq \eps/3$ for all $\alpha$.  Then for the compact set $K:= K_1 \times K_2 \times K_3$ bound $$\mu_\alpha(K^c) \leq \mu_\alpha(K_1^c \times \Omega \times \Omega) +  \mu_\alpha(\Omega \times K_2^c \times \Omega ) +  \mu_\alpha(\Omega \times \Omega \times K_3^c )  \leq \eps\,.$$
\end{proof}

Fact \ref{fact:tight} allows us to reduce the problem of showing tightness of $(g_n,g_n',g_n'')_n$ to showing tightness of each coordinate projection.  For random functions, tightness is well understood.  Define the \emph{modulus of continuity} of a function $h:\R\to \R$ on $[-M,M]$ via $$w_M(h,\delta) := \sup_{\substack{ x,y \in [-M,M] \\ |x - y| \leq \delta } }|h(x) - h(y)|\,.$$

\begin{lemma}[Theorem  $7.3$ of \cite{billingsley-convergence}]\label{lem:tightness}
	A family of random functions $h_n$ on $[-M,M]$ is tight if and only if $\{h_n(0)\}$ is tight and 
	\begin{equation}\label{eq:condition-tight}
	\lim_{\delta \to 0} \limsup_{n \to \infty} \P( w_M(h_n,\delta) \geq \eps) = 0\,.
	\end{equation}
\end{lemma}

Tightness of $(g_n^{(j)}(0) )$ is a consequence of Lemma \ref{lem:fidi} and so we only need to show \eqref{eq:condition-tight}.  We will use the mean value theorem, and so first show that each derivative may be bounded uniformly with high probability.  

\begin{lemma}\label{lem:max-bound}
	For each $j \geq 0, M > 0$ and $n$ we have \begin{align*}
	\E \max_{r \in [-M,M]} |g_n^{(j)}(r)| \leq e^M\,.
	\end{align*}
\end{lemma}
\begin{proof}
	Write \begin{equation}\label{eq:max-expand}
	\E \max_{r \in [-M,M]} |g_n^{(j)}(r)| \leq \sum_{k \geq 0} \frac{\E |g_n^{(j+k)}(0) |}{k!} M^k 
	\end{equation}
	and observe for each $\ell$ \begin{equation}\label{eq:var-expand}
	\E[|g_n^{(\ell)}(0)|^2] = \frac{1}{n^{2\ell + 1}}\sum_{k = 0}^n (k)_\ell^2 \leq 1\,.
	\end{equation}
	%where $(k)_j = k(k-1)\cdots (k-j+1)$ is the falling factorial. 
	Combining lines \eqref{eq:max-expand} and \eqref{eq:var-expand} together with an application of the Cauchy-Schwarz inequality shows the Lemma.
\end{proof}

The condition \eqref{eq:condition-tight} now follows quickly:

\begin{lemma}\label{lem:MOC}
	For $j \in \{0,1,2\}$ and each $M > 0, \eps > 0$ we have \begin{align*}
	\lim_{\delta \to 0} \limsup_{n \to \infty} \P(w_M(g_n^{(j)},\delta) \geq \eps ) = 0\,.
	\end{align*}  
\end{lemma}
\begin{proof}
	By the mean value theorem we have $$w_M(g_n^{(j)},\delta) \leq \delta \max_{r \in [-M,M]} |g_n^{(j+1)}(r)|$$
	and so by Markov's inequality and Lemma \ref{lem:max-bound} $$\P(w_M(g_n^{(j)},\delta) \geq \eps ) \leq \eps^{-1} \delta e^M\,.$$
	Taking $\delta \to 0$ completes the Lemma.
\end{proof}

Convergence in distribution of the processes now follows.

\begin{proof}[Proof of Lemma \ref{lem:convergence-function}]
	By Lemma \ref{lem:fidi}, each family $(g_n^{(j)}(0))$ is tight and so by Fact \ref{fact:tight}, Lemma \ref{lem:tightness} and Lemma \ref{lem:MOC}, the family $(g_n,g_n',g_n'')$ is tight when restricted to $[-M,M]^3$.  Since this family is tight and the finite dimensional distributions converge to those of $(g,g',g'')$ (Lemma \ref{lem:fidi}), we have convergence in distribution (see, e.g.\ \cite[Theorems $5.1$ and $7.1$]{billingsley-convergence}).
\end{proof}

\section{Convergence of zeros} \label{sec:zeros}

Here we show convergence in distribution of the zeros of $g_n$ to those of $g$.  Let $\nu_n$ be the point process corresponding to roots of $g_n$, i.e. $\nu_n(A) = |\{ x \in A : g_n(x) = 0 \}|$ for each Borel set $A$, and similarly define $\nu(A) = |\{x\in A: g(x) = 0 \}$.
To show convergence of $\nu_n$ to $\nu$, we use the following criteria of Kallenberg:
\begin{lemma}[Theorem $2.5$ of \cite{kallenberg}] \label{lem:Kallenberg}
	Suppose $\xi$ is a point process on $\R$ with $\P(\xi(\{x\}) > 0) = 0$ for all $x$ and $\P(\max_{x} \xi(\{x\}) \leq 1 )=1$.  If $\xi_1, \xi_2,\ldots $ are point processes on $\R$ satisfying \begin{enumerate}
		\item \label{it:double} For all bounded intervals $I$ there is a sequence $(\{I_{m,j} : j = 1,\ldots, k_m\})_{m = 1}^\infty$ of partitions of $I$ into intervals with $\lim_{m \to\infty} \max_{1 \leq j \leq k_m} |I_{m,j}| = 0$ such that $$\lim_{m\to\infty} \limsup_{n \to \infty} \sum_{j = 1}^{k_m} \P(\xi_n(I_{m,j}) > 1) = 0\,;$$
		
		\item \label{it:avoid} For each finite union $U$ of bounded intervals $$\P(\xi_n(U) = 0) \to \P(\xi(U) = 0)\,;$$
	\end{enumerate}
	Then $\xi_n$ converges to $\xi$ in distribution with respect to the vague topology.
\end{lemma}

The more difficult condition of the two to show is typically \eqref{it:avoid}.  We would like to say that condition \eqref{it:avoid} of Lemma \ref{lem:Kallenberg} follows from the convergence of $g_n$ to $g$.  Unfortunately, the function $F_0: \mathcal{C}([-M,M]) \to [0,1]$ defined by $F_0(h) = \one\{\exists~x \in [-M,M] : h(x) = 0 \}$ is not continuous with respect to $h$.  Using the convergence of $g_n'$ and $g_n''$ will allow us to successfully approximate $F_0$.  The following basic analytic fact will let us work in that direction.

\begin{lemma}\label{lem:nearly-zero}
	Let $[a,b]$ be an interval.  Then for $\eps$ sufficiently small the following holds: if $h$ is a twice differentiable function so that there is a point $x \in [a+\eps^{1/2},b-\eps^{1/2}]$ with $|h(x)| \leq \eps  $ and $|h'(x)| \geq \eps^{1/4}$ and $\sup_{t \in [a,b]}|h''(t)| \leq \eps^{-1/4}$, then $h$ has a zero in $[a,b]$.
\end{lemma}
\begin{proof}
	Suppose $h(x) \geq 0$ and $h'(x) \geq 0$ as the other cases are symmetric.  Then 
	\begin{align*}
	h(x - \eps^{1/2}) &\leq h(x) - \eps^{1/2} h'(x) + \frac{\eps}{2} \sup_{t \in [a,b]} |h''(t)| \\
	&\leq \eps - \eps^{3/4} + \frac{\eps^{3/4}}{2} \\
	&< 0
	\end{align*}
	for $\eps$ small enough.  Since $x - \eps^{1/2} \in [a,b]$ and $h(x) \geq 0$, there must be some zero in $[x- \eps^{1/2} ,x] \subset [a,b]$.
\end{proof}

Lemma \ref{lem:nearly-zero} can be interpreted as saying that if a function $h$ is small on an interval $[a,b]$ then $h$ has a zero on $[a,b]$ unless something ``atypical'' happens: either $h$ is only small near the boundary of $[a,b]$; the derivative of $h$ is small at all points where $h$ is small; or that the second derivative of $h$ is large.  In our case, these ``atypical'' events truly are atypical in the sense that the probability any happens for $g_n$ or $g$ will tend to zero for large $n$ and small $\eps$.  Further, the functions on $\mathcal{C}([-M,M])$ that underly these events may be easily approximated by continuous functions.  This lets us use Lemma \ref{lem:convergence-function} to prove facts about $g$ from $(g_n)_{n}$ and vice versa.  

In this direction,  we show that the probabilistic upper bound on $g_n^{(j)}$ translates to an identical bound for $g^{(j)}$.

\begin{lemma}\label{lem:g-deriv}
	For each $M > 0$ and $j \in \{0,1,2\}$, $$\P\left(\max_{t \in [-M,M]} |g^{(j)}(t)| \geq T \right) \leq \frac{e^M}{T}\,.$$
\end{lemma}
\begin{proof}
	Note that for each $n$ we have $$\P\left(\max_{t \in [-M,M]} |g_n^{(j)}(t)| \geq T \right) \leq \frac{e^{M}}{T}$$
	for each $T \geq 0$ by Lemma \ref{lem:max-bound} and Markov's inequality.  For each $T$ set $G_{T}:\mathcal{C}([-M,M]) \to [0,1]$ to be the function $$G_{T}(h) = \one\left\{\max_{t \in [-M,M]} |h(t)| \geq T  \right\}\,.$$
	While the function $G_{T}$ is not continuous in the uniform topology, for each $\eps$ we may make a continuous approximation $G_{T}^{(\eps)}$ to $G_{T}$ so that $G_{T+\eps} \leq G_{T}^{(\eps)} \leq G_{T-\eps}$.  Lemma \ref{lem:convergence-function} then implies $$\E G_{T}^{(\eps)}(g_n^{(j)}) \to \E G_{T}^{(\eps)}(g^{(j)})$$
	and so $$\P\left( \max_{t \in [-M,M]} |g^{(j)}(t)| \geq T+\eps \right) \leq \frac{e^M}{T-\eps}$$
	for each $\eps > 0$.  Taking $\eps$ to $0$ completes the proof.
\end{proof}

With Lemma \ref{lem:g-deriv} in tow, we now begin to show that the ``atypical'' events listed in Lemma \ref{lem:nearly-zero} have probability tending to zero as $\eps \to 0$.  We make no attempt of obtaining sharp upper bounds, only those that tend to zero. 

First we show that $g$ is usually not small on short intervals.

\begin{lemma} \label{lem:g-small-ints}
	Fix $M > 0$ and $k > 0$.  For any $k$ intervals $I_j \subset [-M,M]$ with each $|I_j| \leq \eps^{1/2}$, $$\sum_{j = 1}^k \P( \exists~x \in I_j \text{ s.t. }|g(x)|\leq \eps ) = O(\eps^{1/4})$$
	uniformly as $\eps \to 0$.
\end{lemma}
\begin{proof}
	Fix a mesh of points $x_1,\ldots,x_L$ in $I_j$ so that for each $y \in I$ we have $|y - x_j | \leq \eps^{5/4}$ for some $j$ and $L = O(\eps^{-3/4})$.  By Lemma \ref{lem:g-deriv}, we may condition on $|g'(t)| \leq \eps^{-1/4}$ for all $t \in [-M,M]$ and only introduce an event of probability $O(\eps^{1/4}).$  For each $x_j$, we have $\P(|g(x_j)| \leq 2\eps) = O(\eps)$ with implicit constant depending only on $M$.  However, conditioned on $|g'(t)| \leq \eps^{-1/4}$, if there is some $y \in I_j$ with $|g(y)| \leq \eps$, then for the $x_j$ with $|x_j - y| \leq \eps^{5/4}$ we have $$|g(x_j)| \leq \eps + \eps^{5/4} \cdot \eps^{-1/4} = 2\eps$$ by the mean value theorem.  Thus we may bound \begin{align*}
	\sum_{j = 1}^k \P( \exists~x \in I_j \text{ s.t. }|g(x)|\leq \eps ) &\leq O(\eps^{1/4}) + \sum_{j = 1}^L \P(|g(x_j)| \leq 2\eps) \\
	&= O(\eps^{1/4}) + O(\eps^{-3/4})\cdot O(\eps) \\
	&= O(\eps^{1/4})\,.
	\end{align*}
\end{proof}

Next we show that $g$ and $g'$ are typically not simultaneously small.

\begin{lemma}\label{lem:g-and-gp}
	For each $M > 0$ $$\P( \exists~x \in [-M,M] \text{ s.t. }|g(x)| \leq \eps \text{ and }|g'(x)| \leq \eps^{1/4}) = O(\eps^{1/8})\,.$$
\end{lemma}
\begin{proof}
	The proof is very similar to that of Lemma \ref{lem:g-small-ints}. 	Let $x_1,\ldots,x_L$ be points in $[-M,M]$ so that for each $y \in [-M,M]$ we have $|y - x_j| \leq \eps^{9/8}$ for some $j$ and $L = O(\eps^{-9/8})$. Condition on $\max_{t \in [-M,M]} |g'(t)| \leq \eps^{-1/8}$ and $\max_{t \in [-M,M]} |g''(t)| \leq \eps^{-1/8}$.  If there is a point $y \in [-M,M]$ with $|g(y)| \leq \eps$ and $|g'(y)| \leq \eps^{1/4}$ let $x_j$ the point with $|y - x_j| \leq \eps^{9/8}$.  Then $$|g(x_j)| \leq \eps + \eps^{9/8} \eps^{-1/8} \leq 2\eps$$ and $$|g'(x_j)| \leq \eps^{1/4} + \eps^{-1/8}\cdot \eps^{9/8} \leq 2 \eps^{1/4}\,.$$
	
	Since for each point we have $\P(|g(x_j)| \leq 2\eps \text{ and }|g'(x_j)| \leq 2\eps^{1/4}) = O(\eps^{5/4})$, we may bound \begin{align*}
	\P(& \exists~x \in [-M,M] \text{ s.t. }|g(x)| \leq \eps \text{ and }|g'(x)| \leq \eps^{1/4})\\
	&\leq O(\eps^{1/8}) + \sum_{j = 1}^L \P(|g(x_j)| \leq 2\eps \text{ and }|g'(x_j)| \leq 2\eps^{1/4}) \\
	&= O(\eps^{1/8}) + O(\eps^{-9/8}) \cdot O(\eps^{5/4}) \\
	&= O(\eps^{1/8})\,.
	\end{align*} 
\end{proof}

We are now ready to show condition \eqref{it:avoid} of Lemma \ref{lem:Kallenberg}.  Throughout, we will implicitly use approximation arguments as in the proof of Lemma \ref{lem:g-deriv} but are less explicit since the approximation argument is identical. 

\begin{lemma}\label{lem:avoid}
	Let $U$ be a disjoint union of finitely many intervals.  Then $$\P(\nu_n(U) = 0) \to \P(\nu(U) = 0)\,.$$
\end{lemma}
\begin{proof}
	Let $U$ be fixed and take $\eps > 0$.  Define the function $F_0:\mathcal{C}([-M,M]) \to [0,1]$ defined by $$F_0(h) = \one\{\exists~x \in U : h(x) = 0  \}\,.$$
	
	Consider a continuous approximation $F_\eps$ to $F_0$ so that $F_\eps \to F_0$ as $\eps \to 0$ and so that $F_0(h) \leq F_\eps(h) \leq \one\{\exists~x \in U_\eps: |h(x)| \leq \eps  \}$ where $U_\eps = \{x : d(x,U) \leq \eps\}$.  Then by Lemma \ref{lem:convergence-function} we have $$ \E F_\eps(g_n) \to \E F_\eps(g)\,.$$
	
	Then by Lemma \ref{lem:nearly-zero} \begin{align*}
	|\E F_\eps(g_n) - \E F_0(g_n)| &\leq \P(\exists~x \in U_\eps : |g_n(x)| \leq \eps ) - \E F_0(g_n) \\
	&\leq \P\left( \exists~x \in U_\eps : d(x,U^c) \leq \eps^{1/2}, |g_n(x)| \leq \eps \right)  \\
	&\quad+ \P\left(\exists~x \in [-M,M]: |g_n(x)| \leq \eps \text{ and } |g_n'(x)| \leq \eps^{1/4}\right) \\
	&\quad+ \P\left(\max_{x \in [-M,M]} |g_n''(x)| \geq \eps^{-1/4}\right)\,.
	\end{align*}
	
	Taking continuous approximations to the indicator functions of these events shows convergence of these three probabilities with $g_n$ replaced by $g$.  We then have \begin{align*}
	\limsup_{n\to\infty} |\E F_0(g) - \E F_0(g_n)| &\leq \P\left(  \exists~x \in U_\eps : d(x,U^c) \leq \eps^{1/2}, |g(x)| \leq \eps \right)  \\
	&\quad+ \P\left(\exists~x \in [-M,M]: |g(x)| \leq \eps \text{ and } |g'(x)| \leq \eps^{1/4}\right) \\
	&\quad+ \P\left(\max_{x \in [-M,M]} |g''(x)| \geq \eps^{-1/4}\right) \\
	&= O(\eps^{1/8})
	\end{align*}
	by Lemmas \ref{lem:g-deriv}, \ref{lem:g-small-ints} and \ref{lem:g-and-gp}.  Taking $\eps \to 0$ completes the proof.	
\end{proof}

We now show condition \eqref{it:double} of Lemma \ref{lem:Kallenberg}:
\begin{lemma}\label{lem:double}
	For each bounded interval $I$ there is a sequence $(\{I_{m,j} : j = 1,\ldots,k_m \})_{m \geq 1}$ of partitions of $I$ into rectangles with $\lim_{m \to \infty} \max_{1 \leq j \leq k_m} |I_{m,j}| = 0$ such that $$\lim_{m \to \infty} \limsup_{n \to \infty} \sum_{j = 1}^{k_m} \P(\nu_n(I_{m,j}) > 1 ) = 0\,.$$
\end{lemma}
\begin{proof}
	Set $\eps := 1/m$ and define the partition $\{I_{m,j} : j = 1,\ldots,k_m\}$ to be intervals of width $\eps$ with $k_m = O(\eps^{-1})$.  For each interval $I_{m,j}$ let $y_j$ be its midpoint.  Condition on the event that $|g_n'(t)|, |g_n''(t)| \leq \eps^{-1/3}$ for all $t \in [-M,M]$; if there are at least two zeros of $g_n$ in $I_{m,j}$, then there is some zero of $g_n'$ in $I_{m,j}$ and so \begin{align*}
	|g_n(y_j)| &\leq \eps \cdot \eps^{-1/3} = \eps^{2/3} \\
	|g_n'(y_j)| &\leq \eps \cdot \eps^{-1/3} = \eps^{2/3}\,.
	\end{align*} 
	By Lemma \ref{lem:fidi} we have \begin{align*}
\limsup_{n \to \infty} \sum_{j = 1}^{k_m} \P(\nu_n(I_{m,j}) > 1 ) &\leq O(\eps^{1/3}) + \sum_{j = 1}^{k_m}\P\left(|g(y_j)| \leq \eps^{2/3}, |g'(y_j)| \leq \eps^{2/3} \right) \\
&= O(\eps^{1/3}) + O(\eps^{-1})\cdot O(\eps^{4/3}) \\
&= O(\eps^{1/3})\,.
	\end{align*}
	Taking $m \to \infty$ sends $\eps \to 0$, thus completing the proof.
\end{proof}

Convergence of $\{\nu_n\}$ now follows:

\begin{corollary}\label{cor:points}
	As $n \to \infty$ we have $\nu_n \to \nu$.
\end{corollary}
\begin{proof}
	Lemmas \ref{lem:avoid} and \ref{lem:double} show that both conditions of Lemma \ref{lem:Kallenberg} hold and so $\nu_n \to \nu$.
\end{proof}

\section{The limit is not Poisson} \label{sec:not-Poisson}

We now briefly show that the zeros of $g$ do not form a Poisson process.  The classical tool for understanding zeros of a Gaussian process is the Kac-Rice formula which allows us to write factorial moments of $\nu(I)$ as an integral.  To show that $\nu$ is not Poisson, we only need the first two moments.  For proofs and generalizations of the Kac-Rice formula, see \cite[Theorem $3.2$]{azais2009level} and \cite{adler2009random}.

\begin{lemma}[Kac-Rice formula]\label{lem:KR}
	Let $I$ be an interval.  Then \begin{align}
	\E \nu(I) &= \int_I \frac{\E[|g'(x) \,|\,g(x) = 0]}{\sqrt{2\pi \cdot \Var(g(x))}}\,dx \label{eq:KR} \\
	\E[ \nu(I)(\nu(I) - 1)] &=\int_{I^2} \frac{\E\left[| g'(x) \cdot g'(y)| \,|\, g(x) = g(y) = 0 \right]}{ 2\pi \cdot (\det \Sigma )^{1/2}}\,dx\,dy \label{eq:2D-KR}
	\end{align}
	where $\Sigma$ is the covariance matrix of $(g(x),g(y))$.
\end{lemma}

Essentially, $\nu$ is not Poisson because the integrand of \eqref{eq:2D-KR} does not factor into the one-point densities given by the integrand of \eqref{eq:KR}.  In fact, with a bit of work, one can show that the collection of \emph{all} Kac-Rice densities---i.e.\ the integrands for the $k$th factorial moments---uniquely defines $\nu$.  Simply evaluating the integrand of \eqref{eq:2D-KR} at, say, $x = 0$ and $y = 1$ and showing it is not equal to the product of the integrand of \eqref{eq:KR} evaluated at $0$ and $1$ would then show that $\nu$ is not Poisson.  Our approach is similar in spirit to this idea: we consider the probability that there are at least $2$ points in a small interval around zero; this boils down to the fact that at $x = y = 0$ the integrand of \eqref{eq:2D-KR} is $0$, while the integrand of \eqref{eq:KR} is positive at $0$.    

\begin{lemma}\label{lem:not-poisson}
	The process $\nu$ is not a Poisson process.
\end{lemma}
\begin{proof}
	We claim that for $x,y \in I = [-\delta,\delta]$, the integrand of \eqref{eq:2D-KR} converges to zero as $\delta \to 0$.  For the denominator, write \begin{align*}
	\det\left( \Sigma) \right) &= (x - y)^2 \det\Cov\left( \left(g(x), \frac{g(x) - g(y)}{x- y}\right)\right) \\
	&= (x - y)^2(1 + o_{\delta \to 0}(1))\det\Cov\left( \left(g(0), g'(0)\right)\right) \\
	&\geq c (x-y)^2
	\end{align*}
	for some $c > 0$ and $\delta$ small enough.  For the numerator, conditioned on $g(x) = g(y) = 0$ there is some $t \in [x,y]$ so that $g'(t) = 0$ and so we have $|g'(x) g'(y)| \leq (x- y)^2 \max_{r \in [x,y]} |g''(r)|^2$ by the mean value theorem.  This implies \begin{align*}
	\E\left[| g'(x) \cdot g'(y)| \,|\, g(x) = g(y) = 0 \right] &\leq (x-y)^2 \E[\max_{r \in [x,y]} |g''(r)|^2 \,|\, g(x) = g(y) = 0  ] \\
	&\leq C (x - y)^2
	\end{align*} 
	for some $C > 0$ and all, say, $\delta \leq 1$.  We then have the bounds \begin{align}
	\frac{\E\left[| g'(x) \cdot g'(y)| \,|\, g(x) = g(y) = 0 \right]}{ 2\pi \cdot (\det \Sigma )^2} &\leq \frac{C(x - y)^2}{2\pi \sqrt{c} |x - y|} \nonumber \\
	&\leq C' |x - y| \nonumber  \\
	&= O(\delta)\,.\label{eq:KR-bound}
	\end{align}
	
	Let $\lambda$ be the Poisson process with $\E \lambda(U) = \E \nu(U)$ for all Borel $U$.   Then for $I = [-\delta,\delta]$, the Kac-Rice formula shows $\E \nu(I) = \E \lambda(I) = \Theta(\delta)$, and so $$\P(\lambda(I) > 1) = \Theta(\delta^2)\,.$$
	However, we have $$\P(\nu(I) > 1) \leq \E[\nu(I)(\nu(I) - 1)] = \int_{I^2} O(\delta)\,dx\,dy = O(\delta^3)$$
	by \eqref{eq:2D-KR} and \eqref{eq:KR-bound}.  Taking $\delta$ sufficiently small shows that $\nu\neq \lambda$.
\end{proof}

Technically, the conjecture of Shepp and Vanderbei concerns the real roots of $f_n$, not just the positive real roots.  Intuitively, the roots or $f_n$ near $-1$ on the scale of $1/n$ are roughly independent of $g_n$.  Indeed this is the case.

Define the measure $\mu_n$ via $$\mu_n(A) := |\{r \in A : f_n(-1 - r/n) = 0  \}|$$ for each Borel $A$.  We may upgrade Corollary \ref{cor:points} to the pair $(\nu_n,\mu_n)$.  
\begin{lemma}\label{lem:joint}
	As $n \to \infty$, the pair $(\nu_n,\mu_n)$ converges in distribution to $(\nu,\tilde{\nu})$ where $\tilde{\nu}$ is an independent copy of $\nu$.
\end{lemma}

All the legwork for this proof has been done in showing convergence of $\nu_n$; the remaining details for proving Lemma \ref{lem:joint} are discussed in Appendix \ref{app:negative}.  From here, we immediately deduce that the limit $\nu + \tilde{\nu}$ is not Poisson.

\begin{corollary}\label{cor:not-Poisson}
	The process $\nu + \tilde{\nu}$ is not Poisson.
\end{corollary}
\begin{proof}
	By the proof of Lemma \ref{lem:not-poisson}, there is an interval $I$ so that $\nu(I)$ is not Poisson.  By Raikov's theorem, this implies that the convolution $\nu(I) + \tilde{\nu}(I)$ is not Poisson.
\end{proof}

Finally, Theorem \ref{th:main} follows from combining several pieces listed.
\begin{proof}[Proof of Theorem \ref{th:main}]
	Convergence of the point process is shown in Lemma \ref{lem:joint} while the property that it is not-Poisson is shown in Corollary \ref{cor:not-Poisson}.  The existence of a root near $1$ is shown in Lemma \ref{lem:root-near-one}.
\end{proof}

\section{Acknowledgments}
The author thanks Dhruv Mubayi, Will Perkins and Julian Sahasrabudhe for comments on a previous draft.

\appendix
\section{Negative Real Roots} \label{app:negative}
We indicate how we may prove the joint convergence described by Lemma \ref{lem:joint}.  
\subsection{Convergence to a Gaussian process}

Set $h_n(x):=n^{-1/2} f_n(-1 - x/n)$.  First we note that for each $i,j$ and fixed $x,y \in \R$ we have $$\E h_n^{(i)}(x) h_n^{(j)}(y) = \E g_n^{(i)}(x) g_n^{(j)}(y)\,.$$  We compute the cross terms now: \begin{lemma}\label{lem:cross-cov}
	For each fixed $x,y \in \R$ and integers $i,j \geq 0$, we have $$\lim_{n \to \infty}\E g_n^{(i)}(x) h_n^{(j)}(y) = 0\,.$$
\end{lemma} 
\begin{proof}
	Compute \begin{align*}
	\E g_n^{(i)}(x) h_n^{(j)}(y) &= \frac{1}{n^{1 + i + j}} \sum_{r = 0}^n (r)_i (r)_j(-1)^r (1 + x/n)^r (1 + y/n)^r \\
		&= \frac{1}{n^{1 + i +j}}\bigg(\sum_{r = 0}^{n/2} (2r)_i (2r)_j (1 + x/n)^{2r} (1 + y/n)^{2r}  \\
		&\quad- \sum_{r = 0}^{(2n-1)/2} (2r+1)_i (2r+1)_j (1 + x/n)^{2r+1} (1 + y/n)^{2r+1}\bigg) \\
		&= o(1) + \frac{1}{2} \int_0^1 t^{i+j} e^{(x + y)t}\,dt - \frac{1}{2} \int_0^1 t^{i+j} e^{(x + y)t}\,dt \\
		&= o(1)\,.
	\end{align*}
\end{proof}

From here, we see convergence of the pair $(g_n,h_n)$ fairly quickly.  \begin{lemma}
	For each $M$ the function $(g_n,g_n',g_n'',h_n,h_n',h_n'')$ converge in distribution on the space of continuous functions from $[-M,M]^6 \to \R^6$ endowed with the uniform topology to $(g,g',g'',\tilde{g},\tilde{g}',\tilde{g}'')$ where $\tilde{g}$ is an independent copy of $g$.
\end{lemma}
\begin{proof}
	Arguing as in Lemma \ref{lem:fidi} and using Lemma \ref{lem:cross-cov} shows convergence of finite dimensional distributions.  By Lemmas \ref{lem:tightness} and \ref{lem:MOC} along with Fact \ref{fact:tight}, both triples $(g_n,g_n',g_n'')$ and $(h_n,h_n',h_n'')$ are tight; by Fact \ref{fact:tight} we thus have that $(g_n,g_n',g_n'',h_n,h_n',h_n'')$ is tight as well.  Convergence in distribution then follows as in Lemma \ref{lem:convergence-function}.
\end{proof}

\subsection{Convergence to a 2-dimensional point process}

From here, we are ready to show convergence of $(\nu_n,\mu_n) \to (\nu,\tilde{\nu})$.
\begin{proof}[Proof of Lemma \ref{lem:joint}]
	To show convergence of the $2$-dimensional point process $(\nu_n,\mu_n)$, we will show convergence using Lemma \ref{lem:Kallenberg} when both $\nu_n$ and $\mu_n$ are restricted to $[-M,M]$ for each $M > 0$.  It is thus sufficient to show the following two criteria: \begin{enumerate}
		\item \label{it:double-joint} For each $M$ there are sequences $(\{ I_{m,j}^{(i)} :j = 1,\ldots, k_m, i = 1,2 \} )$ of partitions of $[-M,M]$ into intervals with $\lim_{m \to \infty} \max_{1 \leq j \leq k_m} \max_{i \in \{1,2\}} |I^{(i)}_{m,j}| = 0$ such that $$\lim_{m \to \infty} \limsup_{n \to \infty} \left(\sum_{j = 1}^{k_m} \P(\nu_n(I^{(1)}_{m,j} > 1 ))  + \sum_{j = 1}^{k_m} \P(\mu_n(I^{(2)}_{m,j} > 1 )) \right) = 0\,.$$
		\item \label{it:avoid-joint} For each finite union of bounded intervals $U_1, U_2 \subset M$ 
		$$\P(\nu_n(U_1) = \mu_n(U_2) = 0 ) \to \P(\nu(U_1) = 0) \P(\nu(U_2) = 0)\,.$$
	\end{enumerate}  

Criterion \eqref{it:double-joint} follows due to Lemma \ref{lem:double} applying directly for $\nu_n$ and translating identically for $\mu_n$.  Criterion \eqref{it:avoid-joint} similarly may be proved in the same manner as Lemma \ref{lem:avoid}.  In particular, if we set $$F_0(\phi_1,\phi_2) = \one\{ \exists ~x_j \in U_j: \phi_j(x_j) =0  \}$$
we may again approximate from above by a continuous approximation $F_\eps$ and note that $\E F_\eps(g_n,h_n) \to \E F_\eps(g,h)$.  The proof of Lemma \ref{lem:avoid} lists three events, one of which must occur for $F_0$ to differ from $F_\eps$; let these events be called  \emph{exceptional} events.  Directly applying the proof of Lemma \ref{lem:avoid} shows that the probability $g_n$ or $h_n$ exhibits an exception event is $O(\eps^{1/8})$ and so $$\limsup_{n \to \infty} | \E F_0 (g_n,h_n) - \E F_0 (g,\tilde{g})| = O(\eps^{1/8})$$ by a union bound.  Taking $\eps \to 0$ completes the proof.
\end{proof}

\bibliographystyle{abbrv}
\bibliography{Bib}
\end{document}